\documentclass[12pt,a4paper]{article}

\usepackage[utf8]{inputenc}
\usepackage[T1]{fontenc}
\usepackage{amsmath,amssymb,amsthm}
\usepackage{geometry}
\usepackage{xcolor}
\usepackage{authblk}
\usepackage[hidelinks]{hyperref}
\newtheorem{theorem}{Theorem}[section]
\newtheorem{lemma}[theorem]{Lemma}
\newtheorem{proposition}[theorem]{Proposition}
\newtheorem{corollary}[theorem]{Corollary}
\newtheorem{question}[theorem]{Question}

\theoremstyle{definition}
\newtheorem{definition}[theorem]{Definition}
\newtheorem{example}[theorem]{Example}

\theoremstyle{remark}
\newtheorem{remark}[theorem]{Remark}

\newcommand{\ZZ}{\mathbb{Z}}

\newcommand{\calB}{\mathcal{B}}

\newcommand{\Orb}{\mathcal{O}}
\newcommand{\diam}{\operatorname{diam}}

\title{Topological Dynamics of Pullback Maps on Full Shifts}
\author{Alonso Castillo-Ramirez\footnote{Email: alonso.castillor@academicos.udg.mx} \ and Luguis De Los Santos Ba\~{n}os\footnote{Email: luguis.banos@academicos.udg.mx}}
\affil{Departamento de Matemáticas, Centro Universitario de Ciencias Exactas e Ingenier\'ias, Universidad de Guadalajara, Guadalajara, M\'exico.}
\date{\today}

\begin{document}

\maketitle

\begin{abstract}
Let $G$ be a group, let $A$ be a finite alphabet, and let $\phi: G \to G$ be an endomorphism. We study the topological dynamics of the pullback map $\phi^* : A^G \to A^G$, given by $\phi^*(x)=x\circ\phi$, a canonical example of a generalized cellular automaton. In the one-dimensional case, where $G=\mathbb Z$ and $\phi_k(n)=kn$, we prove a sharp dichotomy: $\phi_k^*$ is equicontinuous precisely for $k\in\{-1,0,1\}$, and cofinitely sensitive otherwise. Although the fixed identity coordinate prevents transitivity on the full shift, the restriction to the natural invariant components is topologically mixing exactly when $k\notin\{-1,0,1\}$. We then extend the analysis to countable groups, showing that $\phi^*$ is equicontinuous if and only if every element of $G$ is eventually periodic under $\phi$, while the existence of a non-eventually-periodic element is equivalent to cofinite sensitivity and to the absence of equicontinuous points. Finally, we characterize Bernoulli measure preservation and strong mixing on the punctured configuration space in terms of injectivity and eventual periodicity.
\end{abstract}

\section{Introduction}

Cellular automata are discrete dynamical systems defined by local rules on configuration spaces. In the classical setting, the ambient space is a full shift $A^G$, where $G$ is a group and $A$ is a finite alphabet. The group $G$ acts on $A^G$ by the shift action, and the Curtis-Hedlund-Lyndon theorem characterizes cellular automata as continuous, shift-equivariant self-maps of $A^G$.

A useful generalization is obtained by allowing maps between full shifts over possibly different groups and replacing equivariance with equivariance relative to a group homomorphism \cite{AlonsoAngelAlejandro}. In this framework, every homomorphism $\phi:H\to G$ induces a canonical pullback map
\[
\phi^*:A^G\to A^H,\qquad \phi^*(x)=x\circ\phi.
\]
When $H=G$ and $\phi$ is an endomorphism, $\phi^*$ becomes a self-map of the full shift $A^G$. Although this map has the simplest possible local rule, its dynamics reflect the algebraic behavior of the endomorphism $\phi$.

The purpose of this article is to make this relationship explicit. We focus on standard notions from topological dynamics: equicontinuity, sensitivity, cofinite sensitivity, transitivity, and mixing. The main point is that these dynamical properties are governed by the eventual periodicity of elements of $G$ under iteration of $\phi$.

In the case $G=\ZZ$, every endomorphism $\phi_k$ is multiplication by an integer $k$. We show that the corresponding pullback map $\phi_k^*$ is equicontinuous exactly for $k\in\{-1,0,1\}$ and cofinitely sensitive otherwise. Since the origin is fixed by every endomorphism of $\ZZ$, transitivity cannot hold on the full shift. However, after restricting to the invariant sets determined by the value at the origin, transitivity and mixing hold exactly in the cofinitely sensitive cases.

For a general countable group, two different algebraic conditions arise. Cofinite sensitivity is equivalent to the existence of at least one element with infinite, non-eventually-periodic orbit. By contrast, topological mixing on each invariant component requires that no nontrivial element be eventually periodic. Thus, unlike the one-dimensional integer case, cofinite sensitivity and mixing on invariant components need not be equivalent in general.

In the classical theory of one-dimensional cellular automata, Hedlund (1969) established that a cellular automaton is surjective if and only if it preserves the uniform measure (see Theorem 5.4 in \cite{Hedlund1969}). In a similar spirit, the final section records the parallel measure-theoretic picture for Bernoulli measures. Measure preservation is controlled by the injectivity of the underlying endomorphism $\phi$, while global strong mixing is always obstructed by the fixed identity coordinate. On the punctured configuration space, this obstruction disappears, and strong mixing is again characterized by the absence of eventual periodicity.

\section{Preliminaries}\label{sec:preliminaries}

We begin by fixing the basic terminology used throughout the paper.

\begin{definition}
A \emph{topological dynamical system} is a pair $(X,T)$, where $X$ is a compact metric space and $T:X\to X$ is a continuous map. The iterates of $T$ are denoted by $T^n$, with $T^0=\mathrm{id}_X$ and $T^{n+1}=T\circ T^n$ for all $n\geq 0$.
\end{definition}

\begin{definition}
Let $(X,T)$ be a topological dynamical system. We say that $(X,T)$ is \emph{topologically transitive} if for every pair of nonempty open sets $U,V\subseteq X$, there exists $n>0$ such that
\[
T^n(U)\cap V\neq\emptyset.
\]
We say that $(X,T)$ is \emph{topologically mixing} if for every pair of nonempty open sets $U,V\subseteq X$, there exists $N\geq 0$ such that
\[
T^n(U)\cap V\neq\emptyset\qquad\text{for all }n\geq N.
\]
\end{definition}

\begin{definition}
Let $(X,d)$ be a compact metric space and let $T:X\to X$ be continuous. A point $x\in X$ is called an \emph{equicontinuous point} of $T$ if for every $\varepsilon>0$ there exists $\delta>0$ such that, for every $y\in B_\delta(x)$,
\[
d(T^n(x),T^n(y))<\varepsilon\qquad\text{for all }n\geq 0.
\]
Let $\Omega(T)$ denote the set of equicontinuous points. We say that $T$ is \emph{equicontinuous} if $\Omega(T)=X$, and \emph{almost equicontinuous} if $\Omega(T)$ is residual in $X$.
\end{definition}

\begin{definition}
A topological dynamical system $(X,T)$ is \emph{sensitive to initial conditions}, or simply \emph{sensitive}, if there exists $\varepsilon>0$ such that for every nonempty open set $U\subseteq X$, there exist $x,y\in U$ and $n>0$ satisfying
\[
d(T^n(x),T^n(y))>\varepsilon.
\]
Such an $\varepsilon$ is called a sensitivity constant.
\end{definition}

For a nonempty open set $U\subseteq X$ and $\varepsilon>0$, define
\[
N_T(U,\varepsilon)=\{n\in\ZZ_{\geq 0}:\diam(T^n(U))<\varepsilon\}.
\]

\begin{definition}
A topological dynamical system $(X,T)$ is \emph{cofinitely sensitive} if there exists $\varepsilon>0$ such that, for every nonempty open set $U\subseteq X$, the set $N_T(U,\varepsilon)$ is finite. Equivalently,
\[
\diam(T^n(U))\geq \varepsilon
\]
for all but finitely many $n$.
\end{definition}

Cofinite sensitivity strengthens ordinary sensitivity: separation at a fixed scale does not merely occur at some time, but persists for all sufficiently large times. For one-dimensional cellular automata on full shifts, K\r{u}rka's classification relates sensitivity and almost equicontinuity; see \cite{kuurka2003topological}.

Let $G$ be a group and $A$ a finite alphabet. We write $A^G$ for the full shift over $G$, which is the set of all functions of the form $x : G \to A$.  The left shift action of $G$ on $A^G$ is given by
\[
(gx)(h)=x(g^{-1}h),\qquad g,h\in G,
\]
for every configuration $x\in A^G$.

\begin{definition}\label{def:generalized-cantor-metric}
Let $G$ be a countable group and let $A$ be a finite alphabet. An \emph{exhaustion} of $G$ is a sequence $(E_n)_{n\geq 0}$ of finite subsets of $G$ such that
\[
\emptyset=E_0\subseteq E_1\subseteq E_2\subseteq \cdots,
\qquad
\bigcup_{n\geq 0}E_n=G.
\]
The associated \emph{Cantor metric} on $A^G$ is defined by
\[
d(x,y)=
\begin{cases}
0, & x=y,\\[2mm]
2^{-\max\{n\geq 0 \; : \; x|_{E_n}=y|_{E_n}\}}, & x\neq y.
\end{cases}
\]
\end{definition}

This metric induces the \emph{prodiscrete topology} on $A^G$, which is the product topology obtained by giving $A$ the discrete topology. Since $A$ is finite, the space $A^G$ is compact and metrizable. If $G$ is infinite and $|A|\geq 2$, then $A^G$ is a Cantor space. Different exhaustions define compatible metrics and hence the same topological dynamical notions.

The prodiscrete topology is generated by the cylinders
\[
C(F,q)=\{x\in A^G:x|_F=q\},
\]
where $F\subseteq G$ is a finite subset and $q\in A^F$.

The following definition of generalized cellular automata was introduced in \cite{AlonsoAngelAlejandro}.

\begin{definition}
Let $G$ and $H$ be groups and let $\phi:H\to G$ be a homomorphism. A \emph{$\phi$-cellular automaton} is a map $\tau:A^G\to A^H$ for which there exist a finite subset $S\subseteq G$, called a \emph{memory set}, and a local function $\mu:A^S\to A$ such that
\[
\tau(x)(h)=\mu\big((\phi(h^{-1})x)|_S\big)
\]
for all $x\in A^G$ and $h\in H$.
\end{definition}

The classical definition of a cellular automaton is recovered from the previous one when $G=H$ and $\phi = \text{id}_G$ (cf. \cite[Def. 1.4.1]{CSC10}).

Every homomorphism $\phi:H\to G$ induces a canonical $\phi$-cellular automaton
\[
\phi^*:A^G\to A^H,
\qquad
\phi^*(x)=x\circ\phi,
\]
called the \emph{pullback map}, which has memory set $S=\{e_G\}$ and local function equal to the identity function of $A$.

If $\phi:G\to G$ is an endomorphism and $|A|\geq 2$, then $(A^G,\phi^*)$ is not topologically transitive on the full shift. Indeed,
\[
(\phi^*)^n(x)(e_G)=x(e_G)
\]
for every $x\in A^G$ and every $n\geq 0$, because $\phi(e_G)=e_G$.

%%%%%%%%%%%%%%%%%%%%%%%%%%%%%%%%%%%%%%%%%%%%%

\section{The one-dimensional case}\label{sec:one-dimensional-case}

In this section we take $G=\ZZ$ and let $A$ be a finite alphabet with $|A|\geq 2$. We use the standard exhaustion $E_m=[-m,m]$ of $\ZZ$ and the corresponding Cantor metric. 

\begin{remark}\label{rem:integer-endomorphisms}
Every endomorphism of $\ZZ$ is multiplication by a unique integer $k$. We denote by $\phi_k:\ZZ\to\ZZ$ the map $\phi_k(n)=kn$, and by $\phi_k^*$ the induced pullback map on $A^\ZZ$. Thus
\[
\phi_k^*(x)(i)=x(ki),
\]
and
\[
(\phi_k^*)^n=\phi_{k^n}^*
\]
for all $n\geq 0$.
\end{remark}

Our goal is to classify the dynamics of the maps $\phi_k^*:A^\ZZ\to A^\ZZ$.

We shall use only one elementary symmetry of the full shift. The reversal map
\[
R:A^\ZZ\to A^\ZZ,\qquad R(x)(i)=x(-i),
\]
is a homeomorphism and an isometry for the standard Cantor metric, because each interval $[-m,m]$ is invariant under $i\mapsto -i$. Moreover, $\phi_{-1}^*=R$. This observation is the only symmetry fact needed in the one-dimensional classification.

\begin{remark}\label{rem:combinatorial-equicontinuity}
A configuration $x\in A^\ZZ$ is an equicontinuous point of $\phi_k^*$ if and only if, for every $m\geq 0$, there exists $\ell\geq 0$ such that
\[
x|_{[-\ell,\ell]}=y|_{[-\ell,\ell]}
\quad\Longrightarrow\quad
x|_{k^n[-m,m]}=y|_{k^n[-m,m]}
\]
for all $n\geq 0$.
\end{remark}

\begin{remark}\label{rem:isometry-equicontinuous}
Every isometry of a metric space is equicontinuous.
\end{remark}

\begin{proposition}\label{prop:trivial-a-equicontinuous}
For every $k\in\{-1,0,1\}$, the map $\phi_k^*$ is equicontinuous.
\end{proposition}

\begin{proof}
If $k=1$, then $\phi_1^*$ is the identity. If $k=0$, then $\phi_0^*(x)$ is the constant configuration with value $x(0)$, so the image of any two configurations that agree at the origin is the same. If $k=-1$, then $\phi_{-1}^*=R$, where $R$ is the reversal isometry described above. Hence $\phi_{-1}^*$ is equicontinuous by Remark~\ref{rem:isometry-equicontinuous}.
\end{proof}

\begin{proposition}\label{prop:phi-a-cofinite-sensitive}
For every $k\in\ZZ\setminus\{-1,0,1\}$, the map $\phi_k^*$ is cofinitely sensitive.
\end{proposition}

\begin{proof}
Let $U\subseteq A^\ZZ$ be a nonempty open set. Choose $x\in U$. There exists $\ell>0$ such that
\[
B_{2^{-\ell}}(x)\subseteq U.
\]
Since $|k|>1$, the points $k^n$ are eventually outside $[-\ell,\ell]$ and are pairwise distinct. Choose $N$ such that $|k^n|>\ell$ for all $n\geq N$. Define $y\in A^\ZZ$ by requiring $y|_{[-\ell,\ell]}=x|_{[-\ell,\ell]}$ and
\[
y(k^n)\neq x(k^n)\qquad\text{for all }n\geq N,
\]
with arbitrary values elsewhere. Then $y\in U$. For every $n\geq N$,
\[
(\phi_k^*)^n(x)(1)=x(k^n)\neq y(k^n)=(\phi_k^*)^n(y)(1).
\]
Thus
\[
d((\phi_k^*)^n(x),(\phi_k^*)^n(y))\geq 2^{-1}
\]
for every $n\geq N$. Therefore $\diam((\phi_k^*)^n(U))\geq 2^{-1}$ for all sufficiently large $n$, so $\phi_k^*$ is cofinitely sensitive.
\end{proof}

\begin{corollary}\label{cor:no-eq-points-Z}
If $k\notin\{-1,0,1\}$, then $\phi_k^*$ has no equicontinuous points.
\end{corollary}

\begin{proof}
The argument in Proposition~\ref{prop:phi-a-cofinite-sensitive} can be performed inside every neighborhood of an arbitrary configuration $x$, producing a configuration $y$ whose orbit separates from the orbit of $x$ at the fixed coordinate $1$ for all sufficiently large times.
\end{proof}

If an arbitrary topological dynamical system is equicontinuous, then it is not cofinitely sensitive; however, the converse does not hold in general. Our next result records that this is an equivalence for $\phi_k ^*$.

\begin{corollary}\label{thm:Z-dichotomy}
For $k\in\ZZ$, the map $\phi_k^*$ is equicontinuous if and only if it is not cofinitely sensitive.
\end{corollary}

\begin{proof}
If $k\in\{-1,0,1\}$, then $\phi_k^*$ is equicontinuous by Proposition~\ref{prop:trivial-a-equicontinuous}, so it cannot be cofinitely sensitive. If $k\notin\{-1,0,1\}$, then $\phi_k^*$ is cofinitely sensitive by Proposition~\ref{prop:phi-a-cofinite-sensitive}, and therefore it is not equicontinuous.
\end{proof}

As the value at $0$ is fixed by every $\phi_k^*$, the full shift decomposes into the closed invariant subsets
\[ X_a = \{x\in A^\ZZ: x(0)=a\}, \qquad a \in A. \]

\begin{theorem}\label{thm:Z-transitivity}
Let $k\in\ZZ$. The map $\phi_k^*$ is topologically transitive on $X_a$ for all $a \in A$ if and only if $k\notin\{-1,0,1\}$.
\end{theorem}

\begin{proof}
Fix $a \in A$. Suppose first that $k\notin\{-1,0,1\}$. Let $U,V$ be nonempty open subsets of $X_a$. Choose cylinders
\[
C([-\ell,\ell],p)\cap X_a\subseteq U,
\qquad
C([-r,r],q)\cap X_a\subseteq V,
\]
where $\ell,r\geq 1$ and necessarily $p(0)=q(0)=a$. Choose $N$ so large that $|k|^N r>\ell$. Then the sets $[-\ell,\ell]$ and $k^N([-r,r]\setminus\{0\})$ are disjoint. Define $z\in A^\ZZ$ such that
\[
z|_{[-\ell,\ell]}=p,
\qquad
z(k^Nj)=q(j)\quad\text{for }j\in[-r,r].
\]
Such function $z : \ZZ \to A$ exists because the above conditions only overlap at $j=0$, where they both agree at the value $a$. Then $z\in U$ and $(\phi_k^*)^N(z)\in V$. Hence $\phi_k^*$ is topologically transitive on $X_a$.

Conversely, suppose $k\in\{-1,0,1\}$. If $k=1$, then $\phi_k^*$ is the identity, which is not transitive on $X_a$ because $X_a$ contains disjoint nonempty open sets. If $k=0$, then for every $x\in X_a$, the configuration $\phi_0^*(x)$ is the constant configuration with value $a$. Choose $b\in A\setminus\{a\}$ and set
\[
V=\{x\in X_a:x(1)=b\}.
\]
Then no positive iterate of any point of $X_a$ belongs to $V$, so transitivity fails. If $k=-1$, then $(\phi_{-1}^*)^2=\mathrm{id}$. Choose distinct symbols $b,c\in A$, and let
\[
U=\{x\in X_a:x(1)=b \text{ and } x(-1)=b\},
\qquad
V=\{x\in X_a:x(1)=c\}.
\]
Then $U$ and $V$ are nonempty open subsets of $X_a$, and both $U$ and its reversal are disjoint from $V$. Therefore $(\phi_{-1}^*)^n(U)\cap V=\emptyset$ for all $n\geq 0$, so transitivity fails.
\end{proof}

\begin{corollary}\label{cor:Z-mixing}
The map $\phi_k^*$ is topologically mixing on $X_a$ for all $a \in A$ if and only if $k\notin\{-1,0,1\}$.
\end{corollary}

\begin{proof}
If $k\in\{-1,0,1\}$, then the map is not transitive on $X_a$ by Theorem~\ref{thm:Z-transitivity}, and hence it is not mixing. If $k\notin\{-1,0,1\}$, the proof of Theorem~\ref{thm:Z-transitivity} works for every sufficiently large $n$, because $|k|^n r>\ell$ for all large $n$. Therefore, for every pair of nonempty open sets $U,V\subseteq X_a$, there exists $N$ such that $(\phi_k^*)^n(U)\cap V\neq\emptyset$ for all $n\geq N$.
\end{proof}

%%%%%%%%%%%%%%%%%%%%%%%%%%%%%%%%%
\section{Pullback dynamics over countable groups}\label{sec:countable-groups}

The preceding section raises the question of which parts of the one-dimensional classification remain valid over general groups.

\begin{question}
Let $G$ be a group and let $\phi:G\to G$ be an endomorphism. Consider the pullback map $\phi^*:A^G\to A^G$.
\begin{enumerate}
\item For countable $G$, does the dichotomy between equicontinuity and cofinite sensitivity persist?
\item On the invariant components determined by the value at the identity, when is $\phi^*$ topologically mixing?
\item Is mixing on the invariant components equivalent to cofinite sensitivity in this general setting?
\end{enumerate}
\end{question}

Throughout this section, $G$ is a countable group and $A$ is a finite alphabet with $|A|\geq 2$.

\begin{definition}\label{def:eventually-periodic}
Let $\phi:G\to G$ be an endomorphism. An element $g\in G$ is \emph{periodic} under $\phi$ if there exists $p\geq 1$ such that $\phi^p(g)=g$. It is \emph{eventually periodic} if there exist $\ell\geq 0$ and $p\geq 1$ such that
\[
\phi^{\ell+p}(g)=\phi^\ell(g).
\]
Equivalently, the forward orbit $\Orb(g)=\{\phi^n(g):n\geq 0\}$ is finite.
\end{definition}

\subsection{Equicontinuity and cofinite sensitivity}

We first prove the general form of the equicontinuity--sensitivity dichotomy. The proof uses the freedom to choose a compatible prodiscrete metric, or equivalently an exhaustion of $G$, adapted to the finite forward orbits of $\phi$. Since all such metrics induce the same topology, equicontinuity and sensitivity do not depend on this auxiliary choice.

\begin{lemma}\label{lem:eventual-periodic-contraction}
Let $G$ be countable and let $\phi:G\to G$ be an endomorphism. The following are equivalent:
\begin{enumerate}
\item Every element of $G$ is eventually periodic under $\phi$.
\item There exists an exhaustion of $G$ such that
\[
d(\phi^*(x),\phi^*(y))\leq d(x,y)
\]
for all $x,y\in A^G$.
\end{enumerate}
\end{lemma}

\begin{proof}
Assume first that every element of $G$ is eventually periodic. Enumerate $G$ as $g_1,g_2,\ldots$, with $g_1=e_G$. Since every forward orbit is finite, define
\[
E_0=\emptyset,
\qquad
E_n=\bigcup_{i=1}^n \Orb(g_i)
\quad (n\geq 1).
\]
Then $(E_n)$ is an exhaustion of $G$ and $\phi(E_n)\subseteq E_n$ for every $n$. If $x$ and $y$ agree on $E_n$, then $\phi^*(x)$ and $\phi^*(y)$ also agree on $E_n$, because $\phi(E_n)\subseteq E_n$. Therefore $d(\phi^*(x),\phi^*(y))\leq d(x,y)$.

Conversely, suppose that such an exhaustion exists. By iterating the inequality, we have
\[
d((\phi^*)^n(x),(\phi^*)^n(y))\leq d(x,y)
\]
for all $n\geq 1$. Fix $m\geq 1$ and choose $x,y\in A^G$ such that $x|_{E_m}=y|_{E_m}$ and $x(g)\neq y(g)$ for every $g\in G\setminus E_m$. The inequality implies that $(\phi^*)^n(x)$ and $(\phi^*)^n(y)$ agree on $E_m$ for every $n\geq 1$. Hence, for each $g\in E_m$, the point $\phi^n(g)$ must lie in $E_m$ for every $n\geq 1$. Since $E_m$ is finite, the sequence $(\phi^n(g))_{n\geq 0}$ is eventually periodic. As $\bigcup_m E_m=G$, every element of $G$ is eventually periodic.
\end{proof}

\begin{theorem}\label{thm:general-dichotomy}
Let $G$ be a countable group and let $\phi:G\to G$ be an endomorphism. The following statements are equivalent:
\begin{enumerate}
\item There exists $g\in G$ which is not eventually periodic under $\phi$.
\item The map $\phi^*$ has no equicontinuous points.
\item The map $\phi^*$ is cofinitely sensitive.
\end{enumerate}
Consequently, $\phi^*$ is equicontinuous if and only if every element of $G$ is eventually periodic under $\phi$.
\end{theorem}

\begin{proof}
Assume that $g\in G$ is not eventually periodic. Then the orbit $\Orb(g)$ is infinite. Choose $M$ such that $g\in E_M$, and set $\varepsilon=2^{-M}$.

We first show that there are no equicontinuous points. Let $x\in A^G$ and let $C(E_m,x|_{E_m})$ be an arbitrary basic neighborhood of $x$. Since $\Orb(g)$ is infinite, it eventually avoids the finite set $E_m$. Choose $N$ such that $\phi^n(g)\notin E_m$ for all $n\geq N$. Define $y\in A^G$ so that $y|_{E_m}=x|_{E_m}$ and
\[
y(\phi^n(g))\neq x(\phi^n(g))
\qquad\text{for all }n\geq N,
\]
extending arbitrarily elsewhere. Then $y$ belongs to the chosen neighborhood of $x$, but for every $n\geq N$,
\[
(\phi^*)^n(x)(g)=x(\phi^n(g))\neq y(\phi^n(g))=(\phi^*)^n(y)(g).
\]
Thus $d((\phi^*)^n(x),(\phi^*)^n(y))\geq \varepsilon$ for all $n\geq N$, so $x$ is not an equicontinuous point. Since $x$ was arbitrary, $\phi^*$ has no equicontinuous points.

The same construction proves cofinite sensitivity. Given any nonempty open set $U$, choose a cylinder $C(E_m,q)\subseteq U$. Pick $x\in C(E_m,q)$ and construct $y\in C(E_m,q)$ as above. Then $x,y\in U$ and
\[
\diam((\phi^*)^n(U))\geq \varepsilon
\]
for all sufficiently large $n$. Hence $\phi^*$ is cofinitely sensitive.

Conversely, suppose every element of $G$ is eventually periodic. By Lemma~\ref{lem:eventual-periodic-contraction}, there exists a compatible prodiscrete metric for which $\phi^*$ is a contraction. Hence $\phi^*$ is equicontinuous. In particular, it has equicontinuous points and it cannot be cofinitely sensitive. This proves the equivalences.
\end{proof}

\begin{example}\label{ex:matrix}
Let
\[
M=\begin{pmatrix}1&3\\2&4\end{pmatrix}\in GL_2(\mathbb{Q}),
\]
and let $M$ act on the additive group $G=\mathbb{Q}^2$. A direct computation shows that the orbit of $(1,0)^T$ under $M$ is infinite. Consequently, $(1,0)^T$ is not eventually periodic. Hence, by Theorem~\ref{thm:general-dichotomy}, the induced pullback system is cofinitely sensitive.
\end{example}

Cofinite sensitivity always implies sensitivity, but the converse does not hold in general. For instance, Sturmian subshifts are sensitive but not cofinitely sensitive; see \cite{Moothathu_2007}.

\subsection{Mixing on invariant components}\label{sec:Mixing on invariant components}

Since $\phi(e_G)=e_G$, the coordinate at $e_G$ is fixed by $\phi^*$. We therefore consider, for each $c\in A$, the closed invariant subset
\[
X_c=\{x\in A^G:x(e_G)=c\}.
\]
The next result shows that mixing on these components is governed by a stronger algebraic condition than cofinite sensitivity.

\begin{lemma}\label{lem:finite-escape}
Let $\phi:G\to G$ be an endomorphism. The following are equivalent:
\begin{enumerate}
\item No nontrivial element of $G$ is eventually periodic under $\phi$.
\item For every finite set $F\subseteq G\setminus\{e_G\}$ and every finite set $E\subseteq G$, there exists $N\geq 0$ such that
\[
\phi^n(F)\cap E=\emptyset
\qquad\text{for all }n\geq N.
\]
\end{enumerate}
\end{lemma}

\begin{proof}
Assume that no nontrivial element is eventually periodic. Then, for each $g\in F$, the orbit $\Orb(g)$ is infinite. Since $E$ is finite, the sequence $(\phi^n(g))_{n\geq 0}$ can meet $E$ only finitely many times; otherwise it would repeat some element of $E$ and the orbit of $g$ would be eventually periodic. Since $F$ is finite, a common $N$ works for every $g\in F$.

Conversely, suppose that some $g\neq e_G$ is eventually periodic. Then $\Orb(g)$ is finite. Taking $F=\{g\}$ and $E=\Orb(g)$, we have $\phi^n(F)\cap E\neq\emptyset$ for every $n\geq 0$, so the finite escape condition fails.
\end{proof}

\begin{theorem}\label{thm:general-mixing}
Let $\phi:G\to G$ be an endomorphism. For each $c\in A$, the system $(X_c,\phi^*)$ is topologically mixing if and only if no nontrivial element of $G$ is eventually periodic under $\phi$.
\end{theorem}

\begin{proof}
Suppose first that there exists a nontrivial eventually periodic element $g\in G$.

If the eventual cycle of $g$ contains a nontrivial element, then there exist $k\geq 0$ and $p\geq 1$ such that $h=\phi^k(g)\neq e_G$ and $\phi^p(h)=h$. Choose distinct symbols $a,b\in A$ and consider
\[
U=\{x\in X_c:x(h)=a\},
\qquad
V=\{x\in X_c:x(h)=b\}.
\]
For every $x\in X_c$ and every $m\geq 1$,
\[
(\phi^*)^{mp}(x)(h)=x(\phi^{mp}(h))=x(h).
\]
Therefore $(\phi^*)^{mp}(U)\cap V=\emptyset$ for all $m\geq 1$, and the system is not mixing.

If the orbit of $g$ eventually reaches $e_G$, then there exists $N\geq 0$ such that $\phi^n(g)=e_G$ for all $n\geq N$. Choose $d\in A\setminus\{c\}$ and set
\[
V=\{x\in X_c:x(g)=d\}.
\]
For every $x\in X_c$ and every $n\geq N$,
\[
(\phi^*)^n(x)(g)=x(e_G)=c.
\]
Hence $(\phi^*)^n(V)\cap V=\emptyset$ for all $n\geq N$, so the system is not mixing.

Conversely, assume that no nontrivial element is eventually periodic. Let $U,V\subseteq X_c$ be nonempty open sets. Choose cylinders
\[
C(F_1,q_1)\cap X_c\subseteq U,
\qquad
C(F_2,q_2)\cap X_c\subseteq V,
\]
where, after removing the identity coordinate if necessary, we may assume $F_1,F_2\subseteq G\setminus\{e_G\}$. By Lemma~\ref{lem:finite-escape}, there exists $N\geq 0$ such that
\[
F_1\cap \phi^n(F_2)=\emptyset
\qquad\text{for all }n\geq N.
\]
Moreover, $\phi^n(F_2)$ does not contain $e_G$, since otherwise an element of $F_2$ would eventually reach $e_G$ and hence would be eventually periodic. For each $n\geq N$, define $x_n\in A^G$ by imposing
\[
x_n(e_G)=c,
\qquad
x_n|_{F_1}=q_1,
\qquad
x_n(\phi^n(g))=q_2(g)\quad (g\in F_2),
\]
and extending arbitrarily elsewhere. These conditions are compatible because the hypothesis of the theorem also implies that $\phi$ is injective: otherwise a nontrivial element
of $\ker(\phi)$ would be eventually periodic. Therefore every iterate $\phi^n$ is also injective. Then $x_n\in U$ and $(\phi^*)^n(x_n)\in V$. Therefore $(\phi^*)^n(U)\cap V\neq\emptyset$ for all $n\geq N$, proving mixing.
\end{proof}

\begin{remark}\label{rem:mixing-vs-cofinite}
For $G=\ZZ$, mixing on the components $X_a$ and cofinite sensitivity are equivalent for the family $\{\phi_k^*\}_{k\in\ZZ}$. This equivalence does not persist over arbitrary groups. By Theorems~\ref{thm:general-dichotomy} and \ref{thm:general-mixing}, cofinite sensitivity requires the existence of at least one non-eventually-periodic element, while mixing on $X_a$ requires that no nontrivial element be eventually periodic.
\end{remark}

\begin{remark}\label{rem:full-shift-symmetries}
The discussion above does not require a separate theory of full-shift symmetries. We only use the following standard observation: if $\alpha\in\operatorname{Aut}(G)$ and $f\in\operatorname{Sym}(A)$, then the maps $x\mapsto x\circ\alpha$ and $x\mapsto f\circ x$ are homeomorphisms of $A^G$. Consequently, conjugating a self-map of $A^G$ by either of these homeomorphisms preserves equicontinuity, sensitivity, topological transitivity, and topological mixing. This explains, for example, the role of reversal in the one-dimensional case, but it is not needed as a separate ingredient in the proofs.
\end{remark}

%%%%%%%%%%%%%%%%%%%%%%%%%%%%%%%%%%%%%%%%%%

\section{Bernoulli measures and strong mixing}\label{sec:bernoulli}

In this section we record the measure-theoretic analogue of the preceding results. Throughout the section, $G$ is a countable group, $A$ is a finite alphabet with $|A|\geq 2$, and $p:A\to[0,1]$ is a strictly positive probability distribution. Let $\mu_p$ denote the Bernoulli measure on $A^G$ defined on elementary cylinders by
\[
\mu_p(C(F,q))=\prod_{g\in F}p(q(g)),
\]
where $F\subseteq G$ is finite and $q\in A^F$.

For a finite set $F\subseteq G$ and $E\subseteq A^F$, we also write
\[
C(F,E)=\{x\in A^G:x|_F\in E\}.
\]
The finite cylinders form an algebra generating the Borel $\sigma$-algebra $\calB$ of the prodiscrete topology. We shall use the standard uniqueness theorem for product measures on this cylinder algebra; equivalently, this follows from Carath\'{e}odory's Extension Theorem \cite{FollandRA}.

\begin{theorem}\label{thm:bernoulli-invariance}
Let $\phi:G\to G$ be an endomorphism. The Bernoulli measure $\mu_p$ is invariant under $\phi^*$ if and only if $\phi$ is injective.
\end{theorem}

\begin{proof}
Suppose first that $\phi$ is not injective. Choose distinct $g_1,g_2\in G$ such that $\phi(g_1)=\phi(g_2)=h$. Let $a,b\in A$ be distinct. Since $p$ is strictly positive, the cylinder
\[
C=\{x\in A^G:x(g_1)=a,\ x(g_2)=b\}
\]
has positive measure $p(a)p(b)$. However,
\[
(\phi^*)^{-1}(C)=\{x\in A^G:x(h)=a \text{ and }x(h)=b\}=\emptyset.
\]
Thus $\mu_p((\phi^*)^{-1}(C))\neq \mu_p(C)$, so $\mu_p$ is not invariant.

Conversely, suppose that $\phi$ is injective. Let $D=C(F,q)$ be an elementary cylinder, with $F=\{g_1,\ldots,g_k\}$. Then
\[
(\phi^*)^{-1}(D)=\{x\in A^G:x(\phi(g_i))=q(g_i)\text{ for }i=1,\ldots,k\}.
\]
Because $\phi$ is injective, the elements $\phi(g_1),\ldots,\phi(g_k)$ are distinct. Hence
\[
\mu_p((\phi^*)^{-1}(D))=\prod_{i=1}^k p(q(g_i))=\mu_p(D).
\]
By the uniqueness of the Bernoulli measure on the cylinder algebra, the equality extends to every Borel set. Therefore $\mu_p$ is invariant under $\phi^*$.
\end{proof}

\begin{definition}
Let $(X,\calB,\mu)$ be a probability space and let $T:X\to X$ be measure-preserving. The system $(X,\calB,\mu,T)$ is \emph{strongly mixing} if for all $B_1,B_2\in\calB$,
\[
\lim_{n\to\infty}\mu(T^{-n}(B_1)\cap B_2)=\mu(B_1)\mu(B_2).
\]
\end{definition}

When $\phi$ is injective, the measure $\mu_p$ is invariant. Nevertheless, the coordinate at $e_G$ gives a nontrivial invariant measurable set: for any $a\in A$, the cylinder $B_a=\{x:x(e_G)=a\}$ satisfies $(\phi^*)^{-1}(B_a)=B_a$ and $0<\mu_p(B_a)<1$. Hence the measure-preserving system is not ergodic, and therefore it is neither weakly mixing nor strongly mixing. The remainder of this section focuses exclusively on the punctured configuration space. Set $G^\circ=G\setminus\{e_G\}$, let $X^\circ=A^{G^\circ}$, and let $\calB^\circ$ be the corresponding Borel $\sigma$-algebra.

\begin{proposition}\label{lem:punctured-strong-mixing}
Let $\phi:G\to G$ be an injective endomorphism and let $\psi=\phi|_{G^\circ}$. The system $(X^\circ,\calB^\circ,\mu_p^\circ,\psi^*)$ is strongly mixing if and only if no element of $G^\circ$ is eventually periodic under $\psi$.
\end{proposition}

\begin{proof}
Since $\phi$ is injective and $\phi(e_G)=e_G$, no element of $G^\circ$ maps to $e_G$. Thus $\psi:G^\circ\to G^\circ$ is well-defined. By the same argument as in Theorem~\ref{thm:bernoulli-invariance}, the Bernoulli measure $\mu_p^\circ$ is invariant under $\psi^*$. 

Assume first that no element of $G^\circ$ is eventually periodic under $\psi$. Let $B_1$ and $B_2$ be finite cylinders in $X^\circ$ with finite supports $F_1,F_2\subseteq G^\circ$. Since no point in $F_1$ is eventually periodic, the finite set $\psi^n(F_1)$ eventually avoids $F_2$. Hence there exists $N\geq 0$ such that
\[
\psi^n(F_1)\cap F_2=\emptyset
\qquad\text{for all }n\geq N.
\]
For $n\geq N$, the cylinders $(\psi^*)^{-n}(B_1)$ and $B_2$ depend on disjoint coordinate sets. Since $\mu_p^\circ$ is a product measure,
\[
\mu_p^\circ((\psi^*)^{-n}(B_1)\cap B_2)
=\mu_p^\circ((\psi^*)^{-n}(B_1))\mu_p^\circ(B_2)
=\mu_p^\circ(B_1)\mu_p^\circ(B_2),
\]
where the last equality uses invariance. Thus the strong mixing identity holds for finite cylinders. Since finite cylinders form a generating $\pi$-system, the identity extends to all Borel sets by the standard monotone class argument.

Conversely, suppose that some $g\in G^\circ$ is eventually periodic under $\psi$. Let $h=\psi^k(g)$ and $T\geq 1$ be such that $\psi^T(h)=h$. Choose distinct $a,b\in A$ and define
\[
B_1=\{x\in X^\circ:x(h)=a\},
\qquad
B_2=\{x\in X^\circ:x(h)=b\}.
\]
For every $m\geq 1$,
\[
(\psi^*)^{-mT}(B_1)\cap B_2=\emptyset,
\]
because the intersection would require the coordinate $h$ to have both values $a$ and $b$. Hence along the subsequence $mT$ the measures are equal to zero, whereas
\[
\mu_p^\circ(B_1)\mu_p^\circ(B_2)=p(a)p(b)>0.
\]
Therefore the strong mixing limit cannot hold.
\end{proof}

\begin{corollary}\label{cor:punctured-ergodicity}
If $\phi:G\to G$ is injective and $e_G$ is the only eventually periodic element of $G$ under $\phi$, then the punctured system $(X^\circ,\calB^\circ,\mu_p^\circ,\psi^*)$ is ergodic.
\end{corollary}

\begin{proof}
By Proposition~\ref{lem:punctured-strong-mixing}, the punctured system is strongly mixing. Strong mixing implies ergodicity; see, for example, \cite{hawkins2021ergodic}.
\end{proof}

\section{Conclusions and further directions}\label{sec:conclusions}

We have studied pullback maps $\phi^*:A^G\to A^G$ induced by group endomorphisms. Although these maps have the simplest possible local rule, their dynamics are governed by the algebraic behavior of the forward orbits of $\phi$.

For $G=\ZZ$, the multiplier $k$ completely determines the dynamics: the maps with $k\in\{-1,0,1\}$ are equicontinuous, while all others are cofinitely sensitive and mixing on the invariant components determined by the value at the origin. For countable groups, the corresponding classification separates into two algebraic conditions. Equicontinuity is equivalent to eventual periodicity of all elements of $G$, while cofinite sensitivity is equivalent to the existence of at least one non-eventually-periodic element. Mixing on the invariant components is stronger: it requires the absence of nontrivial eventually periodic elements.

The measure-theoretic results show a parallel but distinct picture. A strictly positive Bernoulli measure is preserved exactly when the endomorphism is injective. Strong mixing is impossible on the full shift because the identity coordinate is fixed, but it holds on the punctured space exactly when the punctured dynamics have no eventually periodic elements.

Several questions remain open. The most natural next step is to study general $\phi$-cellular automata with nontrivial memory sets and local rules. It would also be interesting to determine when ordinary sensitivity implies cofinite sensitivity within broader classes of generalized cellular automata, and to extend the analysis to compact non-discrete alphabets.

\section*{Acknowledgments}

The second author was supported by a SECIHTI Postdoctoral Fellowship \emph{Estancias Posdoctorales por M\'exico}, No. I1200/320/2022.  

\section*{Use of artificial intelligence}

During the preparation of this manuscript, the authors used the artificial intelligence tools ChatGPT 5.5 and Gemini 3.1 Pro to assist with language editing, organization of the exposition, and preliminary feedback on the presentation and consistency of some arguments. All mathematical statements, proofs, references, and final editorial decisions were independently checked and verified by the authors, who take full responsibility for the content of the manuscript.

\end{document}